\documentclass[10pt]{article}

\usepackage{graphicx}%
\usepackage{amsmath,amssymb,amsthm,upref,bm}%

\newtheorem{corollary}{Corollary}%
\newtheorem{theorem}{Theorem}%
\begin{document}

\baselineskip=4.4mm

\makeatletter

\newcommand{\E}{\mathrm{e}\kern0.2pt} 
\newcommand{\D}{\mathrm{d}\kern0.2pt}
\newcommand{\RR}{\mathbb{R}}
\newcommand{\CC}{\mathbb{C}}%
\newcommand{\ii}{\kern0.05em\mathrm{i}\kern0.05em}

\renewcommand{\Re}{\mathrm{Re}} 
\renewcommand{\Im}{\mathrm{Im}}

\def\bottomfraction{0.9}

\title{\bf Metaharmonic functions: mean flux theorem, \\ its converse and related
properties}

\author{Nikolay Kuznetsov}

\date{}

\maketitle

%\vspace{2mm}

\begin{center}
Laboratory for Mathematical Modelling of Wave Phenomena, \\ Institute for Problems
in Mechanical Engineering, Russian Academy of Sciences, \\ V.O., Bol'shoy pr. 61,
St. Petersburg 199178, Russian Federation \\ E-mail address:
nikolay.g.kuznetsov@gmail.com
\end{center}

\begin{abstract}
The mean flux theorems are proved for solutions of the Helmholtz equation and its
modified version. Also, their converses are considered along with some other
properties which generalise those that guarantee harmonicity.
\end{abstract}

\setcounter{equation}{0}

%\vspace{6mm}

\section{Introduction}

`Metaharmonic function' is a convenient (though, out of common use) abbreviation for
`solution of the Helmholtz equation', like `harmonic function' is a widely used
equivalent to `solution of the Laplace equation' (see \cite{ABR}, p.~25, about the
origin of the term `harmonic'). Presumably, the Helmholtz equation
\begin{equation}
\nabla^2 u + \lambda^2 u = 0 , \quad \lambda \in \CC , \label{Hh}
\end{equation}
has the next level of complexity comparing with the Laplace equation $\nabla^2 u =
0$ (here and below $\nabla = (\partial_1, \dots , \partial_m)$ is the gradient
operator and $\partial_i = \partial / \partial x_i$), and so the Greek prefix {\it
meta-} (equivalent to the Latin {\it post-}) looks appropriate for characterising
solutions of equation \eqref{Hh}. Presumably, the term was coined by I.~N. Vekua in
his still widely cited article \cite{Ve1}; its English translation can be found as
Appendix~2 in the book \cite{Ve2}. A pdf file of this paper is also available online
at: 

ftp://ftp.math.ethz.ch/hg/EMIS/journals/TICMI/lnt/vol14/vol14.pdf

A few words about early studies of equation \eqref{Hh}. It was briefly considered by
Euler and Lagrange in their treatments of sound propagation and vibrating membranes,
respectively, dating back to 1759. However, it was Helmholtz who initiated its
detailed investigation in his article \cite{Helm} published in 1860 and dealing with
sound waves in a tube with one open end (organ pipe). An essential point of this
paper is the representation formula for solutions of \eqref{Hh} in a domain $D$ (an
open and connected subset of $\RR^m$, if not stated otherwise), which is similar to
Green's representation of harmonic functions. This opened the way to deriving mean
value properties for metaharmonic functions accomplished by Weber in his papers
\cite{W1} and \cite{W2}, where the following formulae for spheres
\begin{equation}
u (x) = \frac{\lambda r}{4 \pi r^2 \sin \lambda r} \int_{\partial B_r (x)} \!\! u \,
\D S_y \quad \mbox{and} \quad u (x) = \frac{1}{2 \pi r J_0 (\lambda r)} \int_{\partial
B_r (x)} \!\! u \, \D S_y \label{We}
\end{equation}
were obtained for $\lambda > 0$ in the three- and two-dimensional case, respectively
(see also \cite{CH}, pp.~288 and 289, respectively). Here and below $J_\nu$ denotes
the Bessel function of order $\nu$ and $B_r (x) = \{ y : |y-x| < r \}$ is the open
ball of radius $r$ centred at $x$ with $\partial B_r (x) \subset D$. In the
$m$-dimensional formula generalising \eqref{We} to arbitrary $m \geq 2$, the
coefficient on the right-hand side is
\begin{equation}
b (\lambda, r) = \frac{1}{\lambda J_{(m-2)/2} (\lambda r)} \left( \frac{\lambda}{2
\pi r} \right)^{m/2} \, ; \label{co}
\end{equation}
see Theorem~4.6 in the recent survey \cite{Ku}, where further references are
given. 

Formulae \eqref{We} and \eqref{co} are somewhat inconvenient, because require extra
explanations at zeros of the Bessel function. Therefore, it is better to write the
mean value formula for spheres in a slightly different form:
\begin{equation}
a^\circ (\lambda r) \, u (x) = \frac{1}{|\partial B_r|} \int_{\partial B_r (x)} \!\!
u (y) \, \D S_y = M^\circ (u, x, r) \label{G}
\end{equation}
(the last equality defines $M^\circ$---the mean value for spheres). Here $|\partial
B_r|$ is the area of $\partial B_r$ equal to $2 \pi^{m/2} r^{m-1} / \Gamma (m/2)$,
\begin{equation}
a^\circ (\lambda r) = \Gamma \left( \frac{m}{2} \right) \frac{J_{(m-2)/2} (\lambda
r)}{(\lambda r / 2)^{(m-2)/2}} \label{a}
\end{equation}
and $\Gamma$ denotes the gamma function. The definition of $J_\nu (z)$ implies that
$a^\circ (z) \to 1$ as $z \to 0$. Since
\begin{equation}
[z^{-\nu} J_\nu (z)]' = - z^{-\nu} J_{\nu+1} (z) \label{diff}
\end{equation}
(see \cite{Wa}, p.~66), the function $a^\circ$ decreases monotonically on $(0,
j_{m/2,1})$ (as usual $j_{\nu,n}$ denotes the $n$th positive zero of $J_\nu$), and
is positive on the smaller interval $(0, j_{(m-2)/2,1})$.

Along with the mean property for spheres described above, metaharmonic functions
have analogous property for balls. As for harmonic functions, one has to integrate
the formula for spheres with respect to $r$ (of course, the coefficient \eqref{co}
must be moved to the left-hand side beforehand). To obtain the result one has to use
formula 1.8.1.21, \cite{PBM}, which involves two different signs. Since both its
forms are used in the paper, we reproduce it and specify which sign is applied here
and which one below:
\begin{equation}
\int_0^x x^{1 \pm \nu} J_\nu (x) \, \D x = \pm x^{1 \pm \nu} J_{\nu \pm 1} (x) +
\frac{2^{1-\nu}}{\Gamma (\nu)} \left\{ \begin{array}{rr} \!\!\! 0 \!\!\! \\ \!\!\! 1
\!\!\! \end{array} \right\} , \quad \left\{ \begin{array}{rr} \!\!\! \Re \nu > -1 \\
\!\!\! \mathrm{arbitrary} \ \nu \!\!\! \end{array} \right\} . \label{PBM}
\end{equation}
Using this formula with the upper sign and $\nu = (m-2)/2$ while integrating the
spherical mean formula over an {\it admissible}\/ ball, that is, its closure
$\overline{B_r (x)} \subset D$, we obtain the following result:
\begin{equation}
\left( \frac{2 \pi r}{\lambda} \right)^{m/2} J_{m/2} (\lambda r) \, u (x) =
\int_{|y| < r} \!\! u (x+y) \, \D y . \label{Rh}
\end{equation}
This equality, as well as \eqref{We} and \eqref{G}, is valid for every $u \in C^2
(D)$, which satisfies \eqref{Hh} in a domain $D \subset \RR^m$, $m \geq 2$, and for
all admissible balls.

Formula \eqref{Rh} can be written in the form
\begin{equation}
a^\bullet (\lambda r) \, u (x) = \frac{1}{|B_r|} \int_{B_r (x)} \!\! u (y) \, \D y =
M^\bullet (u, x, r) \label{G'}
\end{equation}
(the last equality defines $M^\bullet$---the mean value for balls), which is
analogous to \eqref{G}. Here $|B_r|$ is the volume of $B_r$ and
\begin{equation}
a^\bullet (\lambda r) = \Gamma \left( \frac{m}{2} + 1 \right) \frac{J_{m/2} (\lambda
r)}{(\lambda r / 2)^{m/2}} \label{a'}
\end{equation}
is similar to \eqref{a}, and so $a^\bullet$ has the same properties as $a^\circ$.

Let us turn to the notion of {\it flux} having its origin in hydrodynamics. The
so-called zero flux property is an obvious consequence of harmonicity of a function
$u$ in a domain $D \subset \RR^m$, $m \geq 2$. Indeed, {\it if $D'$ is an arbitrary
bounded subdomain of $D$ such that $\overline{D'} \subset D$ and $\partial D'$ is
piecewise smooth, then}
\begin{equation}
\int_{\partial D'} \frac{\partial u}{\partial n_y} \, \D S_y = 0 \, .\label{D}
\end{equation}
Here and below $n$ denotes the exterior normal to smooth (of the class $C^1$) parts
of domains' boundaries. The integral is known as the flux through $\partial D'$
because $u$, interpreted as the velocity potential, describes an irrotational flow
of an inviscid, incompressible fluid in $D \subset \RR^m$, $m = 2,3$. In the absence
of sources and sinks, the influx is equal to outflux for every subdomain $D' \subset
D$ which is expressed by \eqref{D}.

In 1906, B\^ocher \cite{Bo} and Koebe \cite{Ko} independently discovered the
classical converse to the above assertion in two and three dimensions, respectively
(see also \cite{K}, p.~227); its $m$-dimensional version is as follows.

\begin{theorem}[B\^ocher, Koebe]
Let $D$ be a bounded domain in $\RR^m$, $m \geq 2$. If $u$ belonging to $C^0
(\overline D) \cap C^1 (D)$ satisfies \eqref{D} for every admissible ball, then $u$
is harmonic in $D$.
\end{theorem}

To the best author's knowledge, there is no generalisation of equality \eqref{D} to
metaharmonic functions, to say nothing of an assertion analogous to Theorem~1 for
these functions. Therefore, the main aim of this paper is twofold:
\begin{itemize}
\item to find a relation similar to \eqref{D} for the flux of a metaharmonic
function; \item to obtain an analogue of the B\^ocher--Koebe theorem providing a
sufficient condition of metaharmonicity on the basis of this relation.
\end{itemize}
Also, some other assertions about mean value properties of harmonic functions will
be generalised to the case of metaharmonic functions as well as to solutions of the
modified Helmholtz equation (it differs from \eqref{Hh} by the coefficient's sign).

Recently, another characterization of harmonic functions was obtained in \cite{E};
it involves the so-called polynomial mean flow integrals around a single point. It
would be interesting to investigate this property for metaharmonic functions.

\section{Analogues of formula (4) and Theorem 1 \\ for metaharmonic functions}

While the zero flux property of harmonic functions is an immediate consequence of
Green's first identity for the Laplacian, one more step is required to prove the
following {\it mean flux theorem}\/ for metaharmonic functions.

\begin{theorem}
Let $u \in C^2 (D)$ be metaharmonic in a domain $D \subset \RR^m$, $m \geq 2$. Then
the following equality
\begin{equation}
\int_{\partial B_r (x)} \frac{\partial u}{\partial n_y} \, \D S_y = - \lambda^2
\left( \frac{2 \pi r}{\lambda} \right)^{m/2} J_{m/2} (\lambda r) \, u (x)
\label{Dm}
\end{equation}
holds for every admissible ball $B_r (x)$.
\end{theorem}

\begin{proof}
Let us integrate \eqref{Hh} over an admissible ball $B_r (x)$. In view of Green's
first identity for the Laplacian we obtain
\begin{equation}
\int_{\partial B_r (x)} \frac{\partial u}{\partial n_y} \, \D S_y = - \lambda^2
\int_{B_r (x)} \!\! u \, \D y \, . \label{DD}
\end{equation}
Now, \eqref{Rh} yields the required equality.
\end{proof}

It is straightforward to write \eqref{Dm} in the form similar to \eqref{G}, namely:
\begin{equation}
- \lambda \Gamma \left( \frac{m}{2} \right) \frac{J_{m/2} (\lambda r)}{(\lambda r /
2)^{(m-2)/2}} \, u (x) = \frac{1}{|\partial B_r|} \int_{\partial B_r (x)}
\frac{\partial u}{\partial n_y} \, \D S_y = F (u, x, r) \label{F}
\end{equation}
(the last equality defines $F$---the mean flux). Here the coefficient at $u (x)$
tends to zero as $\lambda \to 0$, and so this equality turns into the zero flux
property of harmonic functions in the limit. From \eqref{F} and \eqref{G}, the
relation 
\begin{equation}
J_{(m-2)/2} (\lambda r) F (u, x, r) = - \lambda J_{m/2} (\lambda r) M^\circ (u, x,
r) \label{FM}
\end{equation}
follows. Thus, the mean flux of a metaharmonic function is expressed in terms of its
mean value for spheres unless $\lambda r$ is a zero of either $J_{m/2}$ or
$J_{(m-2)/2}$.

Let us compare \eqref{Dm} with the zero flux property \eqref{D}, for which purpose
we consider when the left-hand side of the former equality vanishes. First, this
takes place when $r = j_{m/2,n} / \lambda$ provided the domain accommodates balls
centred at $x \in D$ that have these values of radius. Second, the flux vanishes for
all spheres centred at $x$ provided $u (x) = 0$. To illustrate the last case, let us
consider the function $|x|^{-1} \sin |x|$, which is metaharmonic on $\RR^3$ with
$\lambda = 1$. Since this function vanishes on every sphere $\partial B_{\pi k}
(0)$, $k=1,2,\dots$, its flux vanishes for each sphere, whose centre lies on
$\partial B_{\pi k} (0)$ for some $k$. For other locations of centre and $r \neq
j_{m/2,n}$, the mean flux of this function is non-zero.

To demonstrate an essential distinction between \eqref{Dm} and the zero flux
property of harmonic functions, let us consider functions metaharmonic on the whole
$\RR^m$, $m \geq 2$, when, in view of self-similarity, it is sufficient to put
$\lambda = 1$. For such a function the flux has following asymptotic behaviour as
$r \to \infty$:
\begin{equation*}
\int_{\partial B_r (x)} \frac{\partial u}{\partial n_y} \, \D S_y = 2^{(m+1) / 2} u
(x) \left( \pi r \right)^{\frac{m-1}{2}} \! \cos \big( r - \pi (m-3) / 4 \big) + O
\left( r^{(m-3)/2} \right) . \label{asym}
\end{equation*}
This is obtained by using the asymptotics of $J_{m/2} (r)$ (see \cite{Wa}, p.~195)
in \eqref{Dm}. The principal term oscillates unless $u (x) = 0$; its amplitude
increases with $r$, whereas the remainder is bounded for $m = 3$ and decays for $m =
2$.

Let us formulate the assertion analogous to Theorem~1, guaranteeing that a function
satisfying a smoothnesss assumption slightly weaker than that in Theorem~1 and the
mean flux property is metaharmonic. It should be mentioned that functions are
assumed to be real-valued in what follows.

\begin{theorem}
Let $D$ be a bounded domain in $\RR^m$, $m \geq 2$, and let a real-valued $u$ belong
to the Sobolev space $W^{k,p}_{loc} (D)$, where $p \in [1, \infty)$ and an integer
$k \geq 2$ are such that $1 < k - m/p$. If the following equality
\begin{equation}
F (u, x, r) = - \lambda \Gamma \left( \frac{m}{2} \right) \frac{J_{m/2} (\lambda
r)}{(\lambda r / 2)^{(m-2)/2}} \, u (x) \label{Ft}
\end{equation}
with some $\lambda > 0$ holds for every $x \in D$ and all $r \in (0, r (x))$, where
$r (x) > 0$ is such that the ball $B_{r (x)} (x)$ is admissible, then $u$ is
metaharmonic in $D$.
\end{theorem}

\begin{proof}
It is well-known (see, for example, \cite{GT}, Chapter 7) that $u \in W^{k,p}_{loc}
(D)$ is a $C^1$-function on $D$ provided $p$ and $k$ satisfy the described
conditions. Hence \eqref{Ft} is well-defined for all $x \in D$ and all $r \in (0, r
(x))$ with $r (x)$ depending on the distance from $x$ to $\partial D$. Since this
equality is equivalent to \eqref{Dm}, we write it as follows:
\begin{equation*}
\int_{|\theta|=1} \frac{\partial u}{\partial \rho} (x + \rho \theta) \, \D
S^{m-1}_\theta = - \lambda \left( 2 \pi \right)^{m/2} (\lambda \rho)^{1 - m/2}
J_{m/2} (\lambda \rho) \, u (x) \, , \ \rho > 0 . 
\end{equation*}
Here $\D S^{m-1}_\theta$ is the area element of the unit sphere $S^{m-1} \subset
\RR^m$ at $\theta$. Integrating the last relation from $0$ to $r$ with respect to
$\rho$ and using formula \eqref{PBM} with the lower sign and $\nu = m/2$, we obtain
\eqref{G}. Thus, our assertion is a consequence of the following one.
\end{proof}

\begin{theorem}
Let $D$ be a bounded domain in $\RR^m$, $m \geq 2$, and let $u \in C^0 (D)$ be
real-valued. If equality \eqref{G} with some $\lambda > 0$ holds for every $x \in D$
and all $r \in (0, r (x))$, where $r (x) > 0$ is such that the ball $B_{r (x)} (x)$
is admissible, then $u$ is metaharmonic in~$D$.
\end{theorem}

By analogy with the case of harmonic functions, this assertion for metaharmonic
functions should be referred to as the converse mean value theorem for spheres. The
author would be surprised if the following its proof had not been published earlier,
but he failed to find it in the literature.

\begin{proof}
First, it is necessary to show that $u$ is smooth and for this purpose we use the
method applied by Mikhlin in his proof of the corresponding theorem for harmonic
functions (see \cite{M}, Chapter~11, \S~7).

Let $D'$ be a subdomain of $D$ whose closure $\overline{D'} \subset D$ is separated
from $\partial D$ by a layer (strip) formed by parts of balls (discs) located within
$D$; each of these balls (discs) has its centre on $\partial D$ and radius $2
\epsilon$ with $\lambda \epsilon < j_{(m-2)/2, 1}$. By $\omega_\epsilon (|y - x|) =
\omega_\epsilon (r)$ we denote the mollifier considered in \cite{M}, Chapter~1,
\S~1. Let $x \in D'$, then, multiplying \eqref{G} by $\omega_\epsilon (r)$, we
obtain
\[ u (x) \, a^\circ (\lambda r) \, |\partial B_r| \, \omega_\epsilon (r) = 
\omega_\epsilon (r) \int_{\partial B_r (x)} \!\! u (y) \, \D S_y \, .
\]
Integration with respect to $r$ over $(0, \epsilon)$ yields
\[ u (x) \, c (\lambda, \epsilon) = \int_{B_\epsilon (x)} \!\! u (y) \, 
\omega_\epsilon (|y - x|) \, \D y = \int_{D} \!\! u (y) \, \omega_\epsilon (|y - x|)
\, \D y \, . 
\]
Here the last equality follows from the fact that $x \in D'$, whereas
$\omega_\epsilon (|y - x|)$ vanishes outside $B_\epsilon (x)$. Moreover,
\[ c (\lambda, \epsilon) = \int_0^\epsilon  a^\circ (\lambda r) \, |\partial
B_r| \, \omega_\epsilon (r) \, \D r > 0 \, ,
\]
because $a^\circ (\lambda r) > 0$ for $r \in (0, \epsilon)$ in view that $\lambda
\epsilon < j_{(m-2)/2, 1}$. Since $\omega_\epsilon$ is infinitely differentiable,
the obtained representation shows that $u \in C^\infty (D')$. By taking $\epsilon$
arbitrarily small, we see that $u \in C^\infty (D)$.

Now we are in a position to show that $u$ is metaharmonic in $D$. Let $x \in D$ and
let $r (x) > 0$ be such that $B_{r} (x)$ is admissible. Since \eqref{G} holds for
all $r \in (0, r (x))$, for any such $r$ \eqref{Rh} holds as well (it follows from
\eqref{G} by integration). Applying the Laplacian to the integral on the right-hand
side of \eqref{Rh}, we obtain
\[ \int_{|y| < r} \!\! \nabla^2_x \, u (x+y) \, \D y = \int_{|y| = r} \!\! \nabla_x \,
u (x+y) \cdot \frac{y}{r} \, \D S_y \, .
\]
Here the equality is a consequence of Green's first formula. By changing variables
this can be written as follows:
\[ r^{m-1} \frac{\partial}{\partial r} \int_{|\theta|=1} \!\! u (x + r \theta) \, 
\D S^{m-1}_\theta = |S^{m-1}| r^{m-1} \frac{\partial}{\partial r} M^\circ (u, x, r)
\, ,
\]
where $M^\circ (u, x, r) = a^\circ (\lambda r) \, u (x)$ and $a^\circ$ is defined by
\eqref{a}. In view of \eqref{diff}, we have that
\[ \frac{\partial}{\partial r} M^\circ (u, x, r) = - \frac{\lambda J_{m/2} (\lambda r)}
{(\lambda r / 2)^{(m-2)/2}} \, u (x) \, .
\]
Combining the above considerations and \eqref{Rh}, we conclude that
\[ \int_{|y| < r} \!\! [ \nabla^2_x \, u + \lambda^2 u ] \, (x+y) \, \D y = 0
\]
for every $x \in D$ and all $r \in (0, r (x))$. Thus, in each $B_r (x)$ there
exists $y (r, x)$ such that $[ \nabla^2 \, u + \lambda^2 u ] \, (y (r, x)) = 0$.
Since $y (r, x) \to x$ as $r \to 0$, we obtain by continuity that $u$ satisfies the
Helmholtz equation at every $x \in D$, thus being metaharmonic in $D$.
\end{proof}

For metaharmonic functions, the following converse of the mean value property for
balls is an immediate consequence of Theorem~4.

\begin{corollary}
Let $D$ be a bounded domain in $\RR^m$, $m \geq 2$, and let $u \in C^0 (D)$ be
real-valued. If equality \eqref{Rh} with some $\lambda > 0$ holds for every $x \in
D$ and all $r \in (0, r (x))$, where $r (x) > 0$ is such that the ball $B_{r (x)}
(x)$ is admissible, then $u$ is metaharmonic in $D$.
\end{corollary}

\begin{proof}
The assumptions made allow us to differentiate the volume mean value equality
\eqref{Rh} with respect to~$r$, thus obtaining
\begin{equation}
\frac{u (x)}{b (\lambda, r)} = \int_{\partial B_r (x)} \!\! u (y) \, \D S_y \quad
\mbox{for all} \ r \in (0, r (x)).
\label{Sm}
\end{equation}
Indeed, it is sufficient to apply the formula $[z^{\nu} J_\nu (z)]' = z^{\nu}
J_{\nu-1} (z)$ (see \cite{Wa}, p.~66) with $\nu = m/2$ to the left-hand side of
\eqref{Rh}. Since \eqref{Sm} is equivalent to \eqref{G}, the assertion follows from
Theorem~4.
\end{proof}

\section{Other characterizations of metaharmonicity}

It is well-known that the mean values for spheres and balls are equal for harmonic
functions. On the other hand, if these mean values are equal for an arbitrary
continuous function and all admissible balls, then this function is harmonic; see
\cite{BR}, where the two-dimensional case is considered, and \cite{Ku}, Theorem~1.8,
for the general case.

The relation between the mean values for spheres and balls is more complicated for
metaharmonic functions, namely:
\begin{equation}
m J_{m/2} (\lambda r) M^\circ (u, x, r) = \lambda r J_{(m-2)/2} (\lambda r)
M^\bullet (u, x, r) \, . \label{MM}
\end{equation}
This immediately follows from \eqref{G}, \eqref{a}, \eqref{G'} and \eqref{a'}. Since
the ratio of coefficients at $M^\circ (u, x, r)$ and $M^\bullet (u, x, r)$ in this
equality tends to $1$ as $\lambda \to 0$, the above mentioned property of harmonic
functions follows from \eqref{MM}. Let us prove the assertion converse to this
equality.

\begin{theorem}
Let $D \subset \RR^m$, $m \geq 2$, be a bounded domain and let $u \in C^0 (D)$ be
real-valued. If equality \eqref{MM} holds for all $x \in D$ and all $r \in (0, r
(x))$, where $r (x) > 0$ is such that the ball $B_{r (x)} (x)$ is admissible, then
$u$ is metaharmonic in~$D$.
\end{theorem}

\begin{proof}
Let $\rho > 0$ be sufficiently small. If $r \in (0, \rho)$, then $M^\bullet (u,x,r)$
is defined for $x$ belonging to an open subset of $D$ depending on the smallness of
$\rho$. Moreover, $M^\bullet (u,x,r)$ is differentiable with respect to $r$ and
\[ \partial_r M^\bullet (u,x,r) = m r^{-1} [ M^\circ (u,x,r) - M^\bullet (u,x,r) ] 
\quad \mbox{for} \ r \in (0, \rho) .
\]
In view of \eqref{MM}, this takes the form:
\[ \frac{\partial_r M^\bullet}{M^\bullet} = \lambda \frac{J_{(m-2)/2} (\lambda r)}
{J_{m/2} (\lambda r)} - \frac{m}{r} = \lambda \frac{J_{m/2}' (\lambda r)} {J_{m/2}
(\lambda r)} - \frac{m}{2 r} \, ,
\]
where the last equality is a consequence of the recurrence formula (\cite{Wa},
p.~45):
\[ J_{\nu-1} (z) = J_{\nu}' (z) + \frac{\nu}{z} J_{\nu} (z) \, .
\]

Since logarithmic derivatives stand on both sides of the equation for $M^\bullet$,
integrating it with respect to $r$ over the interval $(\epsilon, \rho)$ and letting
$\epsilon \to 0$, we arrive at \eqref{G'} with $\rho$ instead of $r$. In particular,
shrinking $B_\epsilon (x)$ to its centre as $\epsilon \to 0$, one obtains that
$M^\bullet (u,x,\epsilon) \to u (x)$ because $u \in C^0 (D)$, and this takes place
for all $x$ belonging to an arbitrary closed subset of $D$. By letting $\epsilon \to
0$ on the right-hand side, one obtains the factor $\Gamma \left( \frac{m}{2} + 1
\right)$ due to the leading term of the power expansion of $J_{m/2}$. Hence for
every $x \in D$ we have that $M^\bullet (u,x,\rho) = a^\bullet (\lambda \rho) \, u
(x)$ for all admissible $\rho$. Now, the assertion that $u$ is metaharmonic in $D$
follows from Corollary~1 because this equality is equivalent to~\eqref{Rh}.

It is worth mentioning that the above calculations are valid under the assumption
that $M^\bullet (u,x,r)$ and $u (x)$ do not vanish as well as $J_{m/2} (\lambda r)$.
Then the general result follows by continuity.
\end{proof}

Along with \eqref{FM} and \eqref{MM}, there is one more linear relation between
averages over spheres and balls, namely:
\begin{equation}
m F (u, x, r) = - \lambda^2 r M^\bullet (u, x, r) \, ; \label{FM'}
\end{equation}
notice that it is equivalent to \eqref{DD}. Like relation \eqref{MM} in Theorem~5,
it guarantees metaharmonicity of $u$.

\begin{theorem}
Let $D \subset \RR^m$, $m \geq 2$, be a bounded domain, and let a real-valued $u
\in C^1 (D) \cap C^0 (\overline D)$ be such that $F (u, \cdot, r)$ and $M^\circ (u,
\cdot, r)$ are $C^2$-functions on the corresponding subdomain provided $r$ is
sufficiently small. If \eqref{FM'} holds for all $x \in D$ and all $r \in (0, r
(x)]$, where $r (x) > 0$ is such that $B_{r (x)} (x)$ is admissible, then $u$ is
metaharmonic in $D$.
\end{theorem}

\begin{proof}
As in the proof of Theorem 4, let us show first that $u$ is a $C^2$-function
locally, for which purpose we transform \eqref{FM'} into an appropriate form.
Changing $r$ to $\rho$, we write it as follows
\begin{equation*}
\int_{|\theta|=1} \frac{\partial u}{\partial \rho} (x + \rho \theta) \, \D
S^{m-1}_\theta = \frac{- \lambda^2}{\rho^{m-1}} \int_{|y| < \rho} \!\! u (x+y) \, \D
y \ \ \mbox{for} \ \rho \in (0, r (x)] 
\end{equation*}
and integrate this with respect to $\rho$ over $(0, r)$ with $r \in (0, r (x)]$ (the
assumptions imposed on $u$ allow us to do this). Then we have
\begin{equation*}
|S^{m-1}| u (x) = \int_{|\theta|=1} \!\! u (x + r \theta) \, \D S^{m-1}_\theta +
\lambda^2 \int_0^r \frac{\D \rho} {\rho^{m-1}} \int_{|y| < \rho} \!\! u
(x+y) \, \D y \, .
\end{equation*}
Changing the order of integration in the last integral and integrating $\rho^{1-m}$
over $(|y|, r)$ in the obtained inner integral, we arrive at
\[ u (x) \frac{|\partial B_r|}{m} = \int_{\partial B_r (x)} \left[ u (y) + \frac{r}
{m-2} \frac{\partial u}{\partial n_y} \right] \D S_y + \frac{\lambda^2 r^{m-1}}
{m-2} \int_{|y| < r} \frac{u (x+y) \, \D y} {|y|^{m-2}} \, .
\]
Here, the last term is a result of integration, whereas the second term in the
square brackets replaces the integrated term containing $\int_{|y| < r} \! u (x+y)
\, \D y$; this integral is just expressed in terms of flux by virtue of relation
\eqref{DD} equivalent to \eqref{FM'}.

Now we write the last equality in the form
\begin{equation}
u (x) = m M^\circ (u, x, r) + \frac{m \, r}{m-2} F (u, x, r) + \frac{m \,
\lambda^2}{(m-2) |S^{m-1}|} \, U (u, x, r) \, , \label{cor}
\end{equation}
which demonstrates that $u$ is a $C^2$-function. Indeed, the first two terms on the
right-hand side have this property by the assumption made about the averages.
Moreover, the integral operator
\begin{equation}
U (u, x, r) = \int_{|y| < r} \frac{u (x+y) \, \D y} {|y|^{m-2}} \label{pot}
\end{equation}
maps $u \in C^1 (D) \cap C^0 (\overline D)$ to $C^2 (D)$. This fact is well-known
for the Newtonian potential (see \cite{V}, p.~292); since $U$ is similar to this
potential, the cited proof is valid for $U$ with minor amendments.

Green's first formula implies that
\begin{equation}
\int_{\partial B_r (x)} \frac{\partial u}{\partial n_y} \, \D S_y = \int_{B_r (x)}
\!\! \nabla^2_y \, u \, \D y = \int_{|y| < r} \!\! \nabla^2_x \, u \, (x+y) \, \D y
\label{Green}
\end{equation}
for every $x \in D$ and all $r \in (0, r (x))$. Combining this and \eqref{DD}, which
is equivalent to \eqref{FM'}, we obtain
\[ \int_{|y| < r} \!\! [ \nabla^2_x \, u + \lambda^2 u ] \, (x+y) \, \D y = 0 \ \ 
\mbox{for every} \ x \in D \ \mbox{and all} \ r \in (0, r (x)) \, .
\]
In the same way as in the proof of Theorem~4, this yields that $u$ is metaharmonic
in~$D$.
\end{proof}

Of course, if we assume from the beginning that $u \in C^2 (D)$, then the proof of
the last theorem reduces just to the concluding few lines beginning with formula
\eqref{Green}. However, the presented proof demonstrates that lesser smoothness of
$u$ is sufficient being combined with another rather weak condition that $M^\circ
(u, \cdot, r)$ and $F (u, \cdot, r)$ are $C^2$-functions for all sufficiently small
values of~$r$. Indeed, in his book \cite{J} published in 1955, John initiated
studies of the question how to recover a function from its mean values over spheres.
In particular, he established (see pp.~86 and 88) that for $m=2,3$ the continuity of
$u$ is guaranteed when $M^\circ (u, \cdot, r) \in C^2 (\RR^m)$. On the other hand,
it is shown in \cite{J} that $M^\circ$ is required to be of the class $C^m$ to imply
the continuity of $u$ in the case of even $m > 3$. Further references on this topic
can be found in the paper \cite{H}. Recovering a function from values of its mean
fluxes over spheres is an open question.

A by-product of the proof of Theorem 6 is the following.

\begin{corollary}
Let $D \subset \RR^m$, $m \geq 2$, be a bounded domain, and let $u \in C^2 (D)$ be
real-valued and metaharmonic in $D$. Then relation \eqref{cor} holds for all $x \in
D$ and all admissible balls with the integral operator $U$ defined by \eqref{pot}.
\end{corollary}

In Theorems 5 and 6, it is shown that each of relations \eqref{MM} and \eqref{FM'}
serves (under appropriate restrictions) as a sufficient condition for
metaharmonicity of a function. Since relations \eqref{FM} and \eqref{cor} have the
same type, it is reasonable to suppose that they are appropriate candidates as
sufficient conditions for metaharmonicity. Like \eqref{FM'}, either of these
relations involves the average flux~$F$. This conjecture will be considered in a
separate paper.

\section{Mean value properties for solutions of the modified \\ Helmholtz equation}

If $\lambda = \ii \mu$ and $\mu > 0$, then equation \eqref{Hh} takes the form
\begin{equation}
\nabla^2 v - \mu^2 v = 0 \label{MHh}
\end{equation}
and is referred to as the modified Helmholtz equation; we use $v$ to denote its
solution to distinguish it from the metaharmonic $u$.

Many properties of $v$ follow from general results about solutions of elliptic
equations in bounded domains. Indeed, the {\it strong maximum principle}\/ of
E.~Hopf (see \cite{GT}, p.~35) holds for $v$ as well as his boundary point lemma
(see \cite{GT}, p.~34). Since no important applications of \eqref{MHh} were known
until the late 1990s, this equation obtained much less attention than \eqref{Hh}.
For example, there is no mention of \eqref{MHh} in the classical book \cite{CH}, but
\eqref{Hh} is investigated in detail under the name of reduced wave equation (see
\cite{CH}, pp.~313--320), because of its role in studies of various time-harmonic
waves.

Let us outline some applications of equation \eqref{MHh} that appeared about 20
years ago. In the pioneering paper \cite{KPES} published in 1998 (see also
\cite{LWW}, \S~5.4), the authors studied the linearized water-wave problem that
describes interaction of a train of oblique waves with infinitely long cylinders
(surface-piercing or totally immersed). In this problem, equation \eqref{MHh} is
considered in a two-dimensional domain exterior to the cylinders' cross-sections;
its solution is unique for certain geometrical arrangements and wave parameters;
outside these regions of uniqueness, examples of non-uniqueness exist. Subsequently,
a version of the boundary integral equation method was applied for obtaining
numerical solutions of the problems that describe diffraction and radiation of
oblique waves by cylinders; see \cite{PPA}. Another application of \eqref{MHh} deals
with the diffusion-limited coalescence; in \cite{BAF}, this problem was reduced to
solving this equation in a triangular domain. Publications on the modified Helmholtz
equation that appeared during the past two decades demonstrate that properties of
the mean flux and average values of its solutions are worth the reader's attention.

In the paper \cite{Po} dating back to 1938, one finds the $m$-dimensional mean value
formula for spheres
\begin{equation}
\tilde a^\circ (\mu r) \, v (x) = \frac{1}{|\partial B_r|} \int_{\partial B_r (x)}
\!\! v (y) \, \D S_y = \widetilde M^\circ (v, x, r) \label{Gtil}
\end{equation}
(the last equality defines $\widetilde M^\circ$), where
\begin{equation}
\tilde a^\circ (\mu r) = \Gamma \left( \frac{m}{2} \right) \frac{I_{(m-2)/2} (\mu
r)}{(\mu r / 2)^{(m-2)/2}} \, . \label{atil}
\end{equation}
Here and below, $I_\nu$ denotes the modified Bessel function of order $\nu$. Formula
\eqref{Gtil} is valid for all $x$ in a bounded domain $D$ provided $B_r (x)$ is
admissible and $v \in C^2 (D)$ satisfies \eqref{MHh} in $D$. (The three- and
two-dimensional versions of this formula were obtained by Weber simultaneously with
\eqref{We} in his papers \cite{W1} and \cite{W2}, respectively; see also \cite{Po},
p.~199.)

Thus, the only distinction between \eqref{Gtil} and \eqref{a} is that $I_{(m-2)/2}$
replaces $J_{(m-2)/2}$. However, this is essential because the former function,
unlike the latter one, is positive and has exponential growth at infinity. Moreover,
in view of the equality 
\[ [z^{-\nu} I_\nu (z)]' = z^{-\nu} I_{\nu+1} (z)
\]
(see \cite{Wa}, p.~79), the function $\tilde a^\circ$ increases monotonically from
one to infinity on the interval $(0, \infty)$. The assertion converse to the mean
value property \eqref{Gtil} is as follows.

\begin{theorem}
Let $D$ be a bounded domain in $\RR^m$, $m \geq 2$, and let $v \in C^0 (D)$ be
real-valued. If equality \eqref{Gtil} with some $\mu > 0$ holds for every $x \in D$
and all $r \in (0, r (x))$, where $r (x) > 0$ is such that the ball $B_{r (x)} (x)$
is admissible, then $v$ satisfies \eqref{MHh} in $D$.
\end{theorem}

The proof is literally the same as that of Theorem~4. Along with relation
\eqref{Gtil}, solutions of \eqref{MHh} satisfy the mean value property for balls. To
obtain it we integrate \eqref{Gtil} with respect to $r$, which yields:
\begin{equation}
\left( \frac{2 \pi r}{\mu} \right)^{m/2} I_{m/2} (\mu r) \, v (x) = \int_{|y| < r}
\!\! v (x+y) \, \D y . \label{Rh'}
\end{equation}
Here, we used formula 1.11.1.5, \cite{PBM}, with the upper sign and $\nu = (m-2)/2$.
Equality \eqref{Rh'} can be written in the form
\begin{equation}
\tilde a^\bullet (\mu r) \, v (x) = \frac{1}{|B_r|} \int_{B_r (x)} \!\! v (y) \, \D
y = \widetilde M^\bullet (v, x, r) \label{G'til}
\end{equation}
(the mean value for balls $\widetilde M^\bullet$ is defined by the second equality)
analogous to \eqref{Gtil}; here
\begin{equation}
\tilde a^\bullet (\mu r) = \Gamma \left( \frac{m}{2} + 1 \right) \frac{I_{m/2} (\mu
r)}{(\mu r / 2)^{m/2}} \, . \label{a'til}
\end{equation}
This function has the same properties as $\tilde a^\circ$. The assertion converse to
the mean value property \eqref{G'til} is also true.

Furthermore, a consequence of \eqref{Gtil}, \eqref{atil}, \eqref{G'til} and
\eqref{a'til} is the relation between the mean values for spheres and balls
analogous to \eqref{MM}:
\begin{equation}
m I_{m/2} (\mu r) \widetilde M^\circ (v, x, r) = \mu r I_{(m-2)/2} (\mu r)
\widetilde M^\bullet (v, x, r) \, . \label{MMtil}
\end{equation}
The proof of the following theorem is analogous to that of Theorem~5 and is even
simpler because the coefficients at $\widetilde M^\circ (v, x, r)$ and $\widetilde
M^\bullet (v, x, r)$ do not vanish.

\begin{theorem}
Let $D \subset \RR^m$, $m \geq 2$, be a bounded domain and let $v \in C^0 (D)$ be
real-valued. If equality \eqref{MMtil} holds for all $x \in D$ and all $r \in (0, r
(x))$, where $r (x) > 0$ is such that the ball $B_{r (x)} (x)$ is admissible, then
$v$ is a solution of \eqref{MHh} in $D$.
\end{theorem}

The next assertion is similar to Theorem 2 and its proof is literally the~same.

\begin{theorem}
Let $v \in C^2 (D)$ satisfy \eqref{MHh} in a domain $D \subset \RR^m$, $m \geq 2$.
Then
\begin{equation}
\int_{\partial B_r (x)} \frac{\partial v}{\partial n_y} \, \D S_y = \mu^2 \left(
\frac{2 \pi r}{\mu} \right)^{m/2} I_{m/2} (\mu r) \, v (x) \label{T2'}
\end{equation}
for every admissible ball $B_r (x)$.
\end{theorem}

Moreover, the converse assertion is similar to Theorem 3 and the following form of
relation \eqref{T2'} is analogous to \eqref{Gtil}:
\begin{equation}
\mu \Gamma \left( \frac{m}{2} \right) \frac{I_{m/2} (\mu r)}{(\mu r / 2)^{(m-2)/2}}
\, v (x) = \frac{1}{|\partial B_r|} \int_{\partial B_r (x)} \frac{\partial
v}{\partial n_y} \, \D S_y = \widetilde F (v, x, r) \label{Ftil}
\end{equation}
(the last equality defines the mean flux $\widetilde F$). Here the coefficient at
$v$ tends to zero as $\mu \to 0$, and so this equality turns into the zero flux
property of harmonic functions in the limit. Furthermore, combining \eqref{Ftil} and
\eqref{G'til}, \eqref{a'til}, one obtains that
\begin{equation*}
m \widetilde F (v, x, r) = \mu^2 r \widetilde M^\bullet (v, x, r) \, .
\end{equation*}
On the basis of this relation a theorem similar to Theorem~6 can be proved;
moreover, the proof yields the following equality
\begin{equation*}
v (x) = m \widetilde M^\circ (v, x, r) - \frac{m \, r}{m-2} \widetilde F (v, x, r) -
\frac{m \, \mu^2}{(m-2) |S^{m-1}|} \, U (v, x, r)
\end{equation*}
analogous to \eqref{cor}.

\end{document}